\documentclass[10pt]{amsart}

\usepackage{color}
\usepackage{amsxtra}
\usepackage{mathtools}
\usepackage{nicefrac}
\mathtoolsset{showonlyrefs} %	enumera solo las ecuaciones que se citan
\usepackage{ mathrsfs }
\usepackage{tikz}
\usetikzlibrary{shadings}

\theoremstyle{definition}
\newtheorem{defi}{Definition}[section]

\theoremstyle{plain}

\newtheorem{teo}[defi]{Theorem}

\newtheorem{lema}[defi]{Lemma}

\theoremstyle{remark}
\newtheorem{remark}[defi]{Remark}

\DeclareMathOperator*{\dive}{div}
\DeclareMathOperator*{\R}{\mathbb{R}}
\title[{The first nontrivial eigenvalue for a system of $p-$Laplacians}]
{The first nontrivial eigenvalue for a system of $p-$Laplacians with Neumann and 
Dirichlet boundary conditions}
\author[L. M. Del Pezzo and J. D. Rossi]
{Leandro M. Del Pezzo and Julio D. Rossi}

\address{Leandro M. Del Pezzo \and Julio D. Rossi
\hfill\break\indent
CONICET and Departamento  de Matem{\'a}tica, FCEyN,
Universidad de Buenos Aires,
\hfill\break\indent Pabellon I, Ciudad Universitaria (1428),
Buenos Aires, Argentina.}

\email{{\tt ldpezzo@dm.uba.ar,
jrossi@dm.uba.ar
}}

\thanks{
Leandro M. Del Pezzo was partially supported by
CONICET PIP 5478/1438  (Argentina) and Julio D. Rossi  was partially supported
by MTM2011-27998, (Spain)}

\begin{document}

\begin{abstract} 
	We deal with the first eigenvalue for a system of two 
	$p-$Laplacians with Dirichlet and Neumann boundary conditions. If $\Delta_{p}w=\dive(|\nabla w|^{p-2}\nabla w)$ stands for the 
	$p-$Laplacian and $\frac{\alpha}{p}+\frac{\beta}{q}=1,$ we consider
	$$
		\begin{cases}
			-\Delta_pu=  \lambda \alpha |u|^{\alpha-2} u|v|^{\beta}
			&\text{ in }\Omega,\\
			-\Delta_q v=  \lambda \beta |u|^{\alpha}|v|^{\beta-2}v
			&\text{ in }\Omega,\\
		\end{cases}
	$$	
	with mixed boundary conditions
	$$
		u=0, \qquad |\nabla v|^{q-2}\dfrac{\partial v}{\partial \nu }=0,
		\qquad \text{on }\partial \Omega.
	$$
	We show that there is a first non trivial eigenvalue that can be characterized by the 
	variational minimization problem
 	$$
 		\lambda_{p,q}^{\alpha,\beta}  = \min
 		\left\{\dfrac{\displaystyle\int_{\Omega}\dfrac{|\nabla u|^p}{p}\, 
 		dx +\int_{\Omega}\dfrac{|\nabla v|^q}{q}\, dx}
        {\displaystyle\int_{\Omega} |u|^\alpha|v|^{\beta}\, dx}
         \colon (u,v)\in \mathcal{A}_{p,q}^{\alpha,\beta}\right\},
 	$$
 	where $$\mathcal{A}_{p,q}^{\alpha,\beta}=\left\{(u,v)\in W^{1,p}_0(\Omega)\times
   	W^{1,q}(\Omega)\colon
    uv\not\equiv0\text{ and }\int_{\Omega}|u|^{\alpha}|v|^{\beta-2}v
    \, dx=0\right\}.$$
 
 	We also study the limit of $\lambda_{p,q}^{\alpha,\beta} $ 
 	as $p,q\to \infty$ assuming that  $\frac{\alpha}{p} \to \Gamma \in (0,1)$, 
 	and  $ \frac{q}{p} \to Q \in (0,\infty)$ as $p,q\to \infty.$ 
 	We find that this limit problem interpolates between the pure Dirichlet 
 	and Neumann cases for a single equation when we take $Q=1$ and the limits 
 	$\Gamma \to 1$ and $\Gamma \to 0$.
\end{abstract}

\maketitle

\centerline{\it Dedicated to Juan Luis  Vazquez, a great mathematician. }

%%%%%%%%%%%%%%%%%%%%%%%%%%%%%%%%%%%%%%%%%%%%%%%%%%%%%%%%%%%%%%%%%%%%%%%%%%%%%%%%%%
\section{Introduction}\label{intro}
%%%%%%%%%%%%%%%%%%%%%%%%%%%%%%%%%%%%%%%%%%%%%%%%%%%%%%%%%%%%%%%%%%%%%%%%%%%%%%%%%%
	Let $\Omega$ be  bounded domain in $\mathbb{R}^N$ with smooth boundary,
	$1<p,q,<\infty,$ and $0<\alpha,\beta$ such that
	\[
		\dfrac{\alpha}{p}+\dfrac{\beta}{q}=1.
	\]
	The aim of this work is to study the following eigenvalue problem
	\begin{equation}\label{eq:problema}
		\begin{cases}
			-\Delta_pu=  \lambda \alpha |u|^{\alpha-2} u|v|^{\beta}
			&\text{ in }\Omega,\\
			-\Delta_q v=  \lambda \beta |u|^{\alpha}|v|^{\beta-2}v
			&\text{ in }\Omega,\\
		\end{cases}
	\end{equation}
	with mixed boundary conditions
	\begin{equation}\label{eq:bc}
		u=0, \qquad |\nabla v|^{q-2}\dfrac{\partial v}{\partial \nu }=0,
		\qquad \text{on }\partial \Omega.
	\end{equation}
	Here $\Delta_{p}w=\dive(|\nabla w|^{p-2} \nabla w)$ is the usual $p-$Laplacian and
	$\frac{\partial}{\partial \nu}$ is the outer normal derivative.
	
	Our first result is a variational characterization of the first non trivial 
	eigenvalue of our problem.
	
	 \begin{teo}\label{teo:autovalor}
	   If  $\beta>1$ and $p\ge N$ or $q>N,$ 
	   then the first non trivial 
	    eigenvalue is given by
        \begin{equation}\label{eq:autovalor}
			\lambda_{p,q}^{\alpha,\beta} \coloneqq \inf
                 \left\{
                     \dfrac{\displaystyle\int_{\Omega}
                     \dfrac{|\nabla u|^p}{p}\, dx +
                      \int_{\Omega}\dfrac{|\nabla v|^q}{q}\, dx}
                      {\displaystyle\int_{\Omega} |u|^\alpha|v|^{\beta}\, dx}
                      \colon
                      (u,v)\in\mathcal{A}_{p,q}^{\alpha,\beta}
                 \right\}
		\end{equation}
        where
        \[
            \mathcal{A}_{p,q}^{\alpha,\beta}\coloneqq\left\{(u,v)\in W^{1,p}_0(\Omega)\times
             W^{1,q}(\Omega)\colon
             uv\not\equiv0\text{ and }\int_{\Omega}|u|^{\alpha}|v|^{\beta-2}v
             \, dx=0\right\}.
        \]
	\end{teo}

	Next we want to study the behaviour of this first non trivial eigenvalue for large values of $p$ and $q$. We look at the limit as $p,q\to\infty$ of 
	$\lambda_{p,q}^{\alpha,\beta} $. To this end, we assume that
	\begin{equation}\label{eq:A}
	 	\tag{A}
	  	\frac{\alpha}{p} \to \Gamma  \in (0,1) \qquad \mbox{and} \qquad  
		 \frac{q}{p} \to Q\in(0,\infty) \qquad \mbox{as } p,q\to \infty.
	\end{equation}
	Observe that, since $\frac{\alpha}p + \frac{\beta}{q} =1,$ 
	we also get the following limit:
	$$
		\frac{\beta}{q} \to 1-\Gamma \quad\mbox{as } p,q\to\infty.
	$$

	\begin{teo}\label{teo:autovalor.limite.intro}
 		Under the assumption \eqref{eq:A},
 		there exists a sequence $\{(p_n,q_n)\}_{n\in\mathbb{N}}$ with
 		$p_n,q_n \to \infty,$ such that
		$$
			 u_{n} \to u_\infty,\quad v_{n} \to v_\infty
			 \quad \mbox{ uniformly in } \overline{\Omega}
			 \mbox{ as } n\to \infty, 
		$$
		where $(u_n,v_n)$ is an eigenfunction
		corresponding to $\lambda_{p_n,q_n}^{\alpha_n,\beta_n}$ normalized according to $\int_{\Omega} |u_n|^\alpha|v_n|^{\beta}\, dx =1$ for all $n\in \mathbb{N}.$
		Moreover, 
		$$
			\left(\lambda_{p,q}^{\alpha,\beta}\right)^{1/p} 
			\to \Lambda_\infty(\Gamma,Q) \coloneqq
			\inf \left\{
			\frac{\displaystyle \max \left\{ 
			\| \nabla w \|_{L^\infty (\Omega)}; \,  
			\| \nabla z \|^Q_{L^\infty (\Omega)}\right\} }{
			\displaystyle \| |w|^{\Gamma} |z|^{(1-
			\Gamma)Q} \|_{L^\infty (\Omega)}}
			\colon(w,z)\in\mathcal{A_{\infty}}\right\}
 		$$
		as $p,q\to \infty.$ Here 
		\[
		\begin{array}{l}
			\displaystyle \mathcal{A}_{\infty}
			\coloneqq \left\{(w,z)\in W^{1,\infty}_0(\Omega)
			 \times W^{1,\infty} (\Omega)\colon
			wz\not\equiv0\mbox{ and } 
			\right.\\
			\displaystyle \left.\qquad\qquad\qquad \qquad \qquad
			\max_{x\in\Omega} |w|^\Gamma |z_+|^{(1-\Gamma)Q}  =  
			 \max_{x\in\Omega} |w|^\Gamma |z_-|^{(1-\Gamma)Q} 
			\right\},	
		\end{array}
		\]
		where $z_+$ and $z_-$ stand for the 
		positive and negative parts of $z$ respectively.

		In addition, this limit $(u_\infty, v_\infty)$ is a solution 
		to the minimization problem for $\Lambda_\infty(\Gamma,Q)$ and 
		a viscosity solution to
		\begin{equation} \label{eq.u}
		\begin{cases}
			\min\left\{-\langle D^2 u\cdot Du,Du\rangle, |Du|-
				\Lambda_\infty(\Gamma,Q) 
				u^{\Gamma} |v_\infty|^{(1-\Gamma)Q}
				\right\}=0  & \mbox{ in }\Omega, \\
				u=0 & \mbox{ on } \partial \Omega,
		\end{cases}
		\end{equation}
		and
		\begin{equation}\label{eq.v}
				\begin{cases}
 					\min\left\{ -\langle D^2v Dv ,Dv\rangle , 
 					|Dv|-\Lambda_\infty(\Gamma,Q)^{\nicefrac{1}{Q}} 
 					u_\infty^{\nicefrac{\Gamma}{Q}} |v|^{1-\Gamma}
 					\right\} 
 					=0 & \text{ in } 
 					\{v>0\}, \\
 					\max\left
 					\{ -\langle D^2vDv,Dv\rangle, -|Dv|+\Lambda_\infty
 					(\Gamma,Q)^{\nicefrac{1}{Q}} 
 					u_\infty^{\nicefrac{\Gamma}{Q}} |v|^{1-\Gamma}
 					\right\}=0  & \text{ in } 
 					\{v<0\}, \\
  					-\langle D^2vDv,Dv\rangle =0 & \text{ in } 
  					\{v=0\},\\
  					\dfrac{\partial v}{\partial \nu}  = 0 & \text{ on } 
 					\partial \Omega.
				\end{cases}
		\end{equation}
	\end{teo}
	
	In the case that $\Omega$ is a ball of radius $R$ (that is, $\Omega = B_R$),
	or when $\Omega$ is a rectangle (that is, 
	$\Omega=(-R,R)\times (-L,L) \subset \R^2$, we assume here that $L\leq R$), 
	we can obtain an explicit value for this limit value, $\Lambda_\infty(\Gamma,Q)$.

	\begin{teo} \label{valor.bola} \
		\begin{enumerate}
			\item When $\Omega$ is a ball of radius $R$ we have
			$$
				\Lambda_\infty(\Gamma,Q) =
				\left(\dfrac{\Gamma+Q(1-\Gamma)}{\Gamma R}
					\right)^{\Gamma} 
				\left( \frac{\Gamma+Q(1-\Gamma)}{Q(1-\Gamma)R} 
				\right)^{(1-\Gamma)Q}.
			$$
			\item When $\Omega$ is the rectangle $(-R,R)\times (-L,L)$
			we get
			$$
				\Lambda_\infty(\Gamma,Q)  = 
				\begin{cases}
					 \left(\dfrac{\Gamma+Q(1-\Gamma)}{\Gamma R}
						\right)^{\Gamma} 
						\left( \dfrac{\Gamma+Q(1-\Gamma)}{Q(1-\Gamma)R} 
						\right)^{(1-\Gamma)Q}&  \mbox{ if } 
						\dfrac{\Gamma R}{Q(1-\Gamma)} \leq L, \\[10pt]
					 \dfrac{1}{(R-L)^{\Gamma}  L^{1-\Gamma}  } ,
						 & \mbox{ if } 
						\dfrac{\Gamma R}{Q(1-\Gamma)} > L.
				\end{cases}
			$$
		\end{enumerate}
	\end{teo}
	
	Remark that the value $\Lambda_\infty(\Gamma,Q)$ for the ball 
	coincides with the one for the rectangle
	(and does not depends on $L$) when $L$ is close to $R$; while for $L$ small the 
	two values differ (and the latter depends on $L$ and goes to $\infty$ as $L\to 0$).

	Note that for the ball, $\Omega =B_R(0)$, when $q=\alpha =p$ 
	(hence $\beta=0$) we have that $p\lambda_{p,p}^{p,0}$ (given by 
	\eqref{eq:autovalor}) is the first eigenvalue 
	for the Dirichlet $p-$Laplacian and for this eigenvalue, it is proved in 
	\cite{JLM} that $\left(p\lambda_{p,p}^{p,0}\right)^{\nicefrac1p}\to
	\nicefrac{1}{R}$ as $p\to \infty $, one over the radius of the largest 
	ball included in $\Omega$. This value corresponds to the value of 
	$\Lambda_\infty(\Gamma,Q) $ computed in Theorem \ref{valor.bola} 
	since in this case $\Gamma =1$ and $Q=1$. Therefore, we can recover 
	the well known result for a single equation with Dirichlet boundary 
	conditions from our results.
	For the Neumann case we have to consider $q=\beta =p$ 
	(and hence $\alpha=0$). Now we have that $p\lambda_{p,p}^{0,p}$ is the first
	non trivial eigenvalue for the Neumann $p-$Laplacian and 
	for this eigenvalue, it is proved in \cite{Espo,RoSaint} that 
	$\left(p\lambda_{p,p}^{0,p}\right)^{\nicefrac1p}\to\nicefrac1{R}$ as $p\to\infty$, 
	that is $2$ over the diameter of $\Omega$.
	In this case in Theorem \ref{valor.bola} we have to take
	$\Gamma =0$ and $Q=1$ that gives again 
	$\Lambda_\infty (0,1) =\nicefrac1R$. Hence, we recover again the known 
	result for a single equation with Neumann boundary conditions.
	Remark that similar limits cases also hold for the case of the rectangle.

	These limit behaviours hold in general. Note that if we take $Q=1$ in 
	the minimization problem for $\Lambda_\infty(\Gamma,Q) $ and 
	then $\Gamma \to 1$ we get
	\[
		\Lambda_\infty(\Gamma,1) \to
		\inf \left\{ \frac{\displaystyle 
		\max \left\{\| \nabla w \|_{L^\infty (\Omega)};   
		\| \nabla z \|_{L^\infty (\Omega)}
		\right\} }{
		\displaystyle \| w  \|_{L^\infty (\Omega)} }\colon
		(w,z)\in\mathscr{B}
		\right\}.
	\]
	where $\mathscr{B}\coloneqq \{(w,z)\in W^{1,\infty}_0
		(\Omega)\times W^{1,\infty} (\Omega)\colon wz\neq0 \}$
 	This limit value coincides with the first eigenvalue for the Dirichlet 
 	problem for the scalar infinity Laplacian 
 	(just take $z\equiv 1$ and $w$ a first eigenfunction for the Dirichlet 
 	problem), see \cite{JLM}.
 	On the other hand when we let $\Gamma \to 0$ (keeping $Q=1$) we obtain
 	$$
		\Lambda_\infty(\Gamma,1)\to
		\inf \left\{ \frac{\displaystyle \max \left
		\{ \| \nabla w \|_{L^\infty (\Omega)}; \,  
		\| \nabla z \|_{L^\infty (\Omega)}\right\} }{
		\displaystyle \| z  \|_{L^\infty (\Omega)} }
		\colon (w,z)\in\mathcal{B}\right\}
 	$$
	where 
	$\mathcal{B}\coloneqq\{(w,z)\in W^{1,\infty}_0(\Omega)\times 
	W^{1,\infty} (\Omega)\colon wz\not\equiv0 \mbox{ and } 
	\max |z_+|  =  \max |z_-|\}$. 
	Hence in this case we obtain the first nontrivial eigenvalue for the Neumann 
	infinity Laplacian (in this case just take $w (x) = \epsilon \, \mbox{dist}(x,\partial \Omega)$ and $z$ a first 
	non trivial eigenfunction for the Neumann problem and then send $\epsilon$ to zero), 
	see \cite{Espo,RoSaint}. We conclude that our eigenvalue limit problem is 
	somehow in between the Dirichlet and the Neumann cases.
	
	Let us end the introduction giving some references and motivation for the 
	analysis of this problem. Concerning the $p-$Laplacian and its properties we quote \cite{dB,Kru,L,Perera,JLV} and references therein.
	The limit of $p-$harmonic functions, that is, 
	of solutions to $-\Delta_p u =-\mbox{div} (|\nabla u|^{p-2} \nabla u)= 0$, 
	as $p\to\infty$ has been extensively studied in the literature 
	(see \cite{BBM} and the survey \cite{ACJ}) and leads naturally to 
	solutions of the infinity Laplacian, given by 
	$-\Delta_{\infty} u = - \nabla u D^2 u (\nabla u)^t=0$. 
	Infinity harmonic functions (solutions to $-\Delta_\infty u =0$) are
	related to the optimal Lipschitz extension problem (see the survey
	\cite{ACJ}) and find applications in optimal transportation, image
	processing and tug-of-war games (see, e.g.,
	\cite{CMS,GAMPR,PSSW,PSSW2} and the references therein).
	Also limits of the eigenvalue problem related to the $p$-Laplacian
	have been exhaustively examined (see \cite{GMPR,JLM,JL}),
	and lead naturally to the infinity Laplacian eigenvalue problem
	$
		\min \left\{ |\nabla u| (x) - \Lambda_\infty u(x)   ,\ 
		- \Delta_{\infty} u (x)
		\right\}=0.
	$
	In fact, it is proved in \cite{JLM,JL}  that the limit as 
	$p\to \infty$ exists both for the eigenfunctions, 
	$u_{p}\to u_\infty$ uniformly, and for 
	the eigenvalues $(\lambda_p)^{\nicefrac1p} 
	\to \Lambda_\infty =\nicefrac1R$, 
	where the pair $u_\infty$, $\Lambda_\infty$ is a
	non trivial solution to the infinity Laplacian eigenvalue problem.
	
	More recently, the limit problem for the fractional $p-$Laplacian has been 
	studied in \cite{CLM, DRSS, DS, LL}.
	
	Eigenvalues for the $p-$Laplacian are related to the asymptotic behaviour of solutions to
	the corresponding evolutions equations, see, for example, \cite{Bonforte,JL1,JL2}.

	Concerning eigenvalues for systems of $p-$Laplacian type there is a rich 
	recent literature, we refer to \cite{BdF,BoRoSa,MaMa,dNP,Z} and references therein.
	The first case in which there is an study of the limit as $p\to \infty$ of 
	eigenvalues for systems of $p-$Laplacians is \cite{BoRoSa} where both 
	equations are subject to Dirichlet boundary conditions.
	
	\medskip
	
	The paper is organized as follows: in Section \ref{Prel} we collect some 
	preliminary results; in Section \ref{1erAutov} we deal with the first 
	eigenvalue to our problem for fixed exponents (in this section we prove 
	Theorem \ref{teo:autovalor}); in Section \ref{sect-lim} we deal with the 
	limit as $p,q\to \infty$ in a variational setting (showing the first part of 
	Theorem \ref{teo:autovalor.limite.intro}); in Section \ref{sect-ex} we 
	compute explicitly the limit eigenvalue in the case of a ball and a rectangle 
	(see Theorem \ref{valor.bola}); finally, in Section 
	\ref{sect-viscosity} we pass the the limit in the equations 
	in the viscosity sense (finishing the proof of Theorem 
	\ref{teo:autovalor.limite.intro}).

%%%%%%%%%%%%%%%%%%%%%%%%%%%%%%%%%%%%%%%%%%%%%%%%%%%%%%%%%%%%%%%%%%%%%%%%%%%%%%%%%%
\section{Preliminaries}\label{Prel}
%%%%%%%%%%%%%%%%%%%%%%%%%%%%%%%%%%%%%%%%%%%%%%%%%%%%%%%%%%%%%%%%%%%%%%%%%%%%%%%%%%
	We begin with some basic facts that will be needed in 
	subsequent sections. 

	\begin{lema}\label{lema:AuxPoincare1}
		Let $\beta>1,$  $p\ge N,$
		and fix $u\in W^{1,p}_0(\Omega)$ such that 
		$u\not\equiv 0.$
		Then
		\[
			\mathcal{A}_{p,q}^{\alpha,\beta}(u)\coloneqq\left\{v \in W^{1,q}(\Omega)
			\colon \displaystyle\int_\Omega
			|u|^{\alpha}|v|^{\beta-2}v\, dx=0	 \right\}
		\]
		is a closed set in $W^{1,q}(\Omega).$
	\end{lema}
	
	\begin{proof}
		Let $\{v_n\}_{n\in\mathbb{N}}\subset \mathcal{A}_{p,q}^{\alpha,\beta}(u)$ and 
		$v\in W^{1,q}(\Omega)$
		such that
		$v_n\to v$ strongly in $ W^{1,q}(\Omega)$.
		Then, up to a subsequence, $|v_n|^{\beta-2}v_n\to|v|^{\beta-2}v$
		strongly in $L^{\frac{q}{\beta-1}}(\Omega).$
		Since $p\ge N,$ by the Sobolev embedding theorem, we have that
		$|u|^{\alpha}\in L^{\frac{q}{q-\beta+1}}(\Omega).$ Therefore
		\begin{equation}\label{eq:auxp2}
			0=\lim_{n\to\infty}
			\int_{\Omega}|u|^{\alpha}|v_n|^{\beta-2}v_n\, dx
			=\int_{\Omega}|u|^{\alpha}|v|^{\beta-2}v\, dx,
		\end{equation}
		and hence $v\in\mathcal{A}_{p,q}^{\alpha,\beta}(u).$
	\end{proof}
	
	\begin{lema}\label{lema:Poincare1}
		Let  $\beta>1,$  $p\ge N$
		and fix $u\in W^{1,p}_0(\Omega)$ such that 
		$u\not\equiv 0.$
		Then 
		there is a positive constant $C$ such that
		\begin{equation}\label{eq:Poincare1}
			\|v\|_{L^q(\Omega)}\le C\|\nabla v\|_{L^q(\Omega)}
		\end{equation}
		for all $v\in\mathcal{A}_{p,q}^{\alpha,\beta}(u).$
	\end{lema}
	
	\begin{proof}
		We argue by contradiction. Suppose that for all $n\in\mathbb{N}$ there
		exists $v_n\in\mathcal{A}_{p,q}^{\alpha,\beta}(u)$ such that $\|v_n\|_{L^q(\Omega)}=1$ and
		\begin{equation}\label{eq:poincare1}
			\|\nabla v_n\|_{L^q(\Omega)}\le\frac1n.
		\end{equation}
		Then $\{v_n\}_{n\in\mathbb{N}}$ is bounded in $W^{1,q}(\Omega).$ Thus,
		using the Sobolev embedding theorem, we have that
		there exist a subsequence, still denoted by $\{v_n\}_{n\in\mathbb{N}},$
		and
		$v\in W^{1,q}(\Omega)$  such that
		\begin{equation}\label{eq:poincare2}
			\begin{aligned}
				&v_n\rightharpoonup v \text{ weakly in }W^{1,q}(\Omega),\\
				&v_n\to v \text{ strongly in }
				L^{q}(\Omega).
			\end{aligned}
		\end{equation}
		Thus $\|v\|_{L^q(\Omega)}=1,$ and  by \eqref{eq:poincare1}, we get
		\[
			\|\nabla v\|_{L^q(\Omega)}\le \liminf_{n\to\infty}
			\|\nabla v_n \|_{L^q(\Omega)}
			\le \lim_{n\to\infty} \dfrac{1}{n}=0.
		\]
		Then $\nabla v\equiv 0$ and hence $v$ is constant since $\Omega$ is
		connected.
		Moreover, since  $v_n\rightharpoonup v $ weakly in $W^{1,q}(\Omega)$
		and
		$\|v_n\|_{W^{1,q}(\Omega)} \to \|v\|_{W^{1,q}(\Omega)},$ we have that
		$v_n\to v$
		strongly in $W^{1,q}(\Omega).$ By Lemma \ref{lema:AuxPoincare1}, we have
		that
		$v\in\mathcal{A}_{p,q}^{\alpha,\beta}(u).$ This is a contradiction because $v$ is a constant.
	\end{proof}
	
	Note that the best constant $C$ for the validity of \eqref{eq:Poincare1} is
	\[
			\dfrac{1}{C(u)}=\min\left\{
			\dfrac{\|\nabla v\|_{L^q(\Omega)}}{\|v\|_{L^q(\Omega)}}
			\colon v\in\mathcal{A}_{p,q}^{\alpha,\beta}(u)\setminus\{0\}\right\}.
	\]

	\begin{lema}\label{lema:ConsantePoincare}
		 Let $\beta> 1,$  $p\ge N$ 
		and $\{u_n\}_{n\in\mathbb{N}}$ a bounded
		sequence in $W^{1,p}_0(\Omega)$  such that $u_n\not\equiv0$
		for all $n\in\mathbb{N}$.
		If
		\[
			\limsup_{n\to\infty}C(u_n)= \infty
		\]
		 then, up to a subsequence, $u_n\rightharpoonup 0$ weakly in
		 $W^{1,p}(\Omega).$
	\end{lema}
	
	\begin{proof}
		We first assume that $C(u_n)\to\infty.$ For all $n\in\mathbb{N},$
		there is
		$v_n\in\mathcal{A}_{p,q}^{\alpha,\beta}(u_n)$ such that $\|v_n\|_{L^q(\Omega)}=1$ and
		\begin{equation}\label{eq:cp1}
			\dfrac{1}{C(u_n)}=\|\nabla v_n\|_{L^q(\Omega)}.
		\end{equation}
		Then $\{v_n\}_{n\in\mathbb{N}}$ is bounded in $W^{1,q}(\Omega).$
		Therefore there
		exist a subsequence $\{v_{n_k}\}_{k\in\mathbb{N}},$ and
		$v\in W^{1,q}(\Omega)$ such that
		\begin{equation}\label{eq:cp2}
			\begin{aligned}
				&v_{n_k}\rightharpoonup v \text{ weakly in }W^{1,q}(\Omega),\\
				&v_{n_k}\to v \text{ strongly in } L^{q}(\Omega),\\
				&|v_{n_k}|^{\beta-2}v_{n_k}\to |v|^{\beta-2}v
				\text{ strongly in }
				L^{\frac{q}{\beta-1}}(\Omega).
			\end{aligned}
		\end{equation}
		Then $\|v\|_{L^q(\Omega)}=1$ and
		\[
			\|\nabla v\|_{L^q(\Omega)}
			\le\liminf_{k\to\infty}\|\nabla v_{n_k}\|_{L^q(\Omega)}
			=\lim_{k\to\infty}\dfrac{1}{C(u_{n_k})}=0.
		\]
		Therefore $v$ is a constant. Moreover,
		since $\|v\|_{L^q(\Omega)}=1,$ we have that
		$ v=\pm\nicefrac{1}{|\Omega|^{\nicefrac1q}}.$
		
		On the other hand, since $\{u_{n_k}\}_{k\in\mathbb{N}}$ is bounded in
		$W^{1,p}(\Omega)$ and $p\ge N,$ there
		exist a subsequence, still denoted $\{u_{n_k}\}_{k\in\mathbb{N}},$ and
		$u\in W^{1,p}(\Omega)$ such that
		\begin{equation}\label{eq:cp3}
			\begin{aligned}
				&u_{n_k}\rightharpoonup u \text{ weakly in }W^{1,p}(\Omega),\\
				&|u_{n_k}|^{\alpha}\to |u|^{\alpha} \text{ strongly in }
				L^{\frac{q}{q-\beta+1}}(\Omega).
			\end{aligned}
		\end{equation}
		
		Using \eqref{eq:cp2} and \eqref{eq:cp3}, we get
		\[
			0=\lim_{k\to\infty}\int_{\Omega}|u_{n_k}|^{\alpha}
			|v_{n_k}|^{\beta-2}v_{n_k}\, dx
			=\int_{\Omega}|u|^{\alpha}|v|^{\beta-2}v\, dx
			=\dfrac{\pm1}{|\Omega|^{\frac{\beta-1}{q\beta}}}
			\int_{\Omega}|u|^{\alpha}\, dx.
		\]
		Therefore $u\equiv0.$
	\end{proof}
	
	The proof of 
	the next lemma is classical and therefore omitted in this paper.
	
	\begin{lema}\label{lema:Poincare2}
		If  $q>N$ then there is a positive constant $C=C(q,N,\Omega)$ such that
		\[
			\|v\|_{L^q(\Omega)}\le C\|\nabla v\|_{L^p(\Omega)}
		\]
		for all $v\in\{w\in W^{1,p}(\Omega)\colon \exists x_0\in\Omega
		\text{ with }
		w(x_0)=0\}.$
	\end{lema}
%%%%%%%%%%%%%%%%%%%%%%%%%%%%%%%%%%%%%%%%%%%%%%%%%%%%%%%%%%%%%%%%%%%%%%%%%%%%%%%%%%%%
\section{The first non trivial eigenvalue}\label{1erAutov}
%%%%%%%%%%%%%%%%%%%%%%%%%%%%%%%%%%%%%%%%%%%%%%%%%%%%%%%%%%%%%%%%%%%%%%%%%%%%%%%%%%%%
	A natural definition of an eigenvalue is a value $\lambda$ for which there is
	$(u,v)\in W^{1,p}_0(\Omega)\times W^{1,p}(\Omega)\setminus\{(0,0)\}$ such
	that
	\begin{equation}\label{eq:ws}
		\begin{aligned}
			\int_{\Omega} |\nabla u|^{p-2}\nabla u\nabla w\, dx
			&=\lambda\alpha\int_{\Omega}|u|^{\alpha-2}u|v|^{\beta}w \,dx, \\
			\int_{\Omega} |\nabla v|^{q-2}\nabla v\nabla z\, dx
			&=\lambda\beta\int_{\Omega}|u|^{\alpha}|v|^{\beta-2}v z\, dx, \\
		\end{aligned}
	\end{equation}
	for all $(w,z)\in W^{1,p}_0(\Omega)\times W^{1,q}(\Omega);$ that is, $(u,v)$ is a nontrivial
	solution of \eqref{eq:problema}--\eqref{eq:bc}. In this context, the pair
	$(u,v)$ is called an eigenfunction corresponding to $\lambda.$ 

    Note that, if $\alpha>1$ then $(u,v)\equiv (0,1)$ is a solution of 
    \eqref{eq:problema}--\eqref{eq:bc} for all
    $\lambda\in\mathbb{R},$ that is every $\lambda\in\mathbb{R}$ is 
    an eigenvalue. We say that a value $\lambda$ is a non trivial 
    eigenvalue if there is
    $(u,v)\in W^{1,p}_0(\Omega)\times W^{1,q}(\Omega)$ such that 
    $uv\not\equiv0$ in $\Omega$ and $(u,v)$ is an eigenfunction 
    corresponding to $\lambda.$ 

    \begin{remark}\label{re:elcero}
		If $(u,v)\in W^{1,p}_0(\Omega)\times W^{1,p}(\Omega)$ is a solution of
        \eqref{eq:problema}--\eqref{eq:bc} with
        $\lambda=0$ then
        \[
	    	\int_{\Omega} |\nabla u|^{p-2}\nabla u\nabla w\, dx=
		    \int_{\Omega} |\nabla v|^{q-2}\nabla v\nabla z\, dx=0
        \]
        for all $(w,z)\in W^{1,p}_0(\Omega)\times W^{1,p}(\Omega).$
        Therefore $u\equiv 0$ and $v$ is constant, that is $0$ is a simple
        eigenvalue.
	\end{remark}

    If $\lambda$ is a non trivial eigenvalue then there is
    $(u,v)\in W^{1,p}_0(\Omega)\times W^{1,q}(\Omega)$ such that $uv\not\equiv0$
    in $\Omega$ and $(u,v)$ is a solution of   \eqref{eq:problema}--\eqref{eq:bc}.
    Then
    \begin{align*}
		\int_{\Omega} |\nabla u|^{p}\, dx
         &=\lambda\alpha\int_{\Omega}|u|^{\alpha}|v|^{\beta}\, dx, \\
		 \int_{\Omega} |\nabla v|^{q}\, dx&=\lambda\beta\int_{\Omega}|u|^{\alpha}
		 |v|^{\beta} \, dx.
	\end{align*}
    Therefore, using that $\frac{\alpha}{p}+\frac{\beta}{q}=1,$ we have
    that
    \begin{equation}\label{eq:autcarac}
		\lambda =\dfrac{\displaystyle\int_{\Omega}\dfrac{|\nabla u|^p}{p}\, dx +
        \int_{\Omega}\dfrac{|\nabla v|^q}{q}\, dx}
         {\displaystyle\int_{\Omega} |u|^\alpha|v|^{\beta}\, dx}\ge0
    \end{equation}
    Moreover, by Remark \ref{re:elcero}, we have $\lambda>0.$

	On the other hand, taking $z\equiv1$ in \eqref{eq:ws}, we get
    \[
	    \int_{\Omega}|u|^{\alpha}|v|^{\beta-2}v \, dx=0.
    \]
    Thus, our candidate for first non trivial eigenvalue is
     \begin{equation}\label{eq:autovalor2}
			\lambda_{p,q}^{\alpha,\beta}  \coloneqq \inf
                 \left\{
                     \dfrac{\displaystyle\int_{\Omega}
                     \dfrac{|\nabla u|^p}{p}\, dx +
                      \int_{\Omega}\dfrac{|\nabla v|^q}{q}\, dx}
                      {\displaystyle\int_{\Omega} |u|^\alpha|v|^{\beta}\, dx}
                      \colon
                      (u,v)\in\mathcal{A}_{p,q}^{\alpha,\beta}
                 \right\}
		\end{equation}
	where
        \[
            \mathcal{A}_{p,q}^{\alpha,\beta}\coloneqq\left\{(u,v)\in 
            W^{1,p}_0(\Omega)\times
             W^{1,q}(\Omega)\colon
             uv\not\equiv0\text{ and }\int_{\Omega}|u|^{\alpha}|v|^{\beta-2}v
             \, dx=0\right\}.
        \]
    
    \subsection{Scaling invariance of $\lambda_1$.}
   		If we take $(u,v)\in \mathcal{A}_{p,q}^{\alpha,\beta}$ such that
       	\begin{equation}\label{normaliz}
       		\int_{\Omega} |u|^\alpha|v|^{\beta}\, dx = 1
       	\end{equation}
       	and we scale both functions according to
       	$$
       		\tilde{u} = a u \qquad \tilde{v} = b v
       	$$
       	we get
		$
       		\int_{\Omega} 
       		|\tilde{u}|^\alpha|\tilde{v}|^{\beta}\, dx = a^\alpha b^\beta$.
       	Then, to still have \eqref{normaliz} we impose
		$a^\alpha b^\beta =1$.
		On the other hand, we have
		$$
			\displaystyle\int_{\Omega}\dfrac{|\nabla \tilde{u}|^p}{p}\, dx +
                      \int_{\Omega}\dfrac{|\nabla \tilde{v}|^q}{q}\, dx =
                      \displaystyle a^p \int_{\Omega}\dfrac{|\nabla u|^p}{p}\, dx + b^q
                      \int_{\Omega}\dfrac{|\nabla v|^q}{q}\, dx := a^p A + b^q B,
		$$
		and then we want to compute
		$$
			\min_{a^\alpha b^\beta =1} a^p A + b^q B.
		$$
		This leads to (using Lagrange's multipliers)
		$
			p a^{p-1} A = \theta \alpha a^{\alpha-1} b^\beta$ and $
			q b^{q-1} B = \theta \beta a^{\alpha} b^{\beta-1}$,
		with
		$a^\alpha b^\beta =1$.
		That is,
		$
			p a^{p} A = \theta \alpha$ and $
			q b^{q} B = \theta \beta$
		and we arrive to
		$$
			\beta p a^{p} A = \alpha q b^{q} B.
		$$
		This computation shows that in a minimizing sequence we can assume that the terms
		$$
			\displaystyle\int_{\Omega}\dfrac{|\nabla {u_n}|^p}{p}\, dx \qquad
			 \mbox{and} \qquad
                      \int_{\Omega}\dfrac{|\nabla {v_n}|^q}{q}\, dx
		$$
		are of the same order. 

    \subsection{Is $\lambda_{p,q}^{\alpha,\beta}$ a non trivial eigenvalue?}
    	We start showing that $\lambda_{p,q}^{\alpha,\beta}$ is not a non trivial
    	eigenvalue when $\alpha =0$ or $\beta=0$.
    	
    	Observe that if $q=\beta$ and $\alpha=0$ then
    	$\lambda_{p,q}^{0,q}\ge \nicefrac{\lambda_q^\mathbf{N}}{q}$
    	where $\lambda_q^\mathbf{N}$ is the first non trivial eigenvalue of the Neumann 
    	$q-$Laplacian
    	that is
    	\[
    		\lambda_q^{\mathbf{N}}=\min\left\{
    		\dfrac{\displaystyle \int_{\Omega }|\nabla v|^q \, dx}{
    		\displaystyle \int_{\Omega }|v|^q \, dx}\colon v\in W^{1,q}(\Omega)
    		\setminus\{0\}, \int_{\Omega}|v|^{q-2}v\,dx=0\right\}.
    	\]
    	Moreover, if $\phi\in C^{1}_{0}(\Omega)$ and $v$ is an eigenfunction
    	corresponding to $\lambda_q^\mathbf{N}$ such that $\phi v\not\equiv0$ 
    	then $(\varepsilon\phi,v)\in\mathcal{A}_{p,q}^{0,q}$
    	for all $\varepsilon>0.$ Then
    	\[
    		\dfrac{\lambda_q^\mathbf{N}}{q}\le \lambda_{p,q}^{0,q}
    		\le\varepsilon^p \dfrac{\displaystyle \int_{\Omega } \frac{|\nabla\phi |^p}{p} \, dx}{
    		\displaystyle \int_{\Omega }|v|^q \, dx} 
    		+\dfrac{\displaystyle \int_{\Omega }\frac{|\nabla v|^q}{q} \, dx}{
    		\displaystyle \int_{\Omega }|v|^q \, dx}
    		=\varepsilon^p \dfrac{\displaystyle \int_{\Omega }\frac{|\nabla\phi|^p}{p} \, dx}{
    		\displaystyle \int_{\Omega }|v|^q \, dx} + \dfrac{\lambda_q^\mathbf{N}}q
    		\quad\forall\varepsilon>0.
    	\]
    	Therefore, passing to the limit as $\varepsilon\to 0$ we have that
    	$\lambda_{p,q}^{0,q}= \nicefrac{\lambda_q^\mathbf{N}}{q}$ 
    	
    	We claim that $\lambda_{p,q}^{0,q}$ is not 
    	a non trivial eigenvalue. Suppose, contrary to our claim, that 
    	$\lambda_{p,q}^{0,q}$ is a non trivial 
    	eigenvalue. Then there exists $(u,v)\in \mathcal{A}_{p,q}^{0,q}$ such that
    	\[
    		\lambda_{p,q}^{0,q}
    		= \dfrac{\displaystyle \int_{\Omega }\frac{|\nabla u |^p}{p} \, dx}{
    		\displaystyle \int_{\Omega }|v|^q \, dx} 
    		+\dfrac{\displaystyle \int_{\Omega } \frac{|\nabla v|^q}{q} \, dx}{
    		\displaystyle \int_{\Omega }|v|^q \, dx}>
    		\dfrac{\displaystyle \int_{\Omega }\frac{|\nabla v|^q}{q} \, dx}{
    		\displaystyle \int_{\Omega }|v|^q \, dx}
    	\]
    	since $u\neq0.$ Therefore $\nicefrac{\lambda_q^\mathbf{N}}{q}=
    	\lambda_{p,q}^{0,q}>\nicefrac{\lambda_q^\mathbf{N}}{q},$ a contradiction that implies that $\lambda_{p,q}^{0,q}$ is not a non trivial eigenvalue.
    	
    	Similarly, if $p=\alpha$ and $\beta=0$ then
    	$\lambda_{p,q}^{p,0}= \nicefrac{\lambda_p^\mathbf{D}}{p}$
    	is not a non trivial eigenvalue. Here $\lambda_p^\mathbf{D}$ is 
    	the first eigenvalue of the  Dirichlet $p-$Laplacian, that is 
    	\[
    		\lambda_p^{\mathbf{D}}=\min\left\{
    		\dfrac{\displaystyle \int_{\Omega }|\nabla u|^p \, dx}{
    		\displaystyle \int_{\Omega }|u|^p \, dx}\colon u\in W^{1,p}_0(\Omega)
    		\setminus\{0\}
    		\right\}.
    	\]

    	Now we show that if $\beta>1$ and $p\ge N$
    	 or $q>N$ then 
    	$\lambda_{p,q}^{\alpha,\beta}$ is the first non trivial eigenvalue.
    	
    	\begin{proof}[Proof of Theorem \ref{teo:autovalor}]
			By \eqref{eq:autcarac} and \eqref{eq:autovalor2}, 
			we only need to prove that $\lambda_{p,q}^{\alpha,\beta} $ is a 
			non trivial eigenvalue. Let $\{(u_n,v_n)\}_{n\in\mathbb{N}}\subset
	    	W_0^{1,p}(\Omega)\times W^{1,q}(\Omega)$ such that
	    	\begin{align}
				\label{eq:autovalor1}\int_{\Omega}|u_n|^{\alpha}|v_n|^{\beta-2}v_n\, dx&=0,\\
				\label{eq:autovalor2.33}\int_{\Omega}|u_n|^{\alpha}|v_n|^{\beta}\, dx&=1,
	    	\end{align}
			and
			\begin{equation}\label{eq:autovalor3}
				\lambda_{p,q}^{\alpha,\beta} =\lim_{n\to\infty}\int_{\Omega}
				\dfrac{|\nabla u_n|^p}{p}\, dx +
				\int_{\Omega}\dfrac{|\nabla v_n|^q}{q}\, dx.
			\end{equation}
			Then, using the Poincare inequality, we have that $\{u_n\}_{n\in\mathbb{N}}$
			is bounded in $W^{1,p}(\Omega).$

			We now split the rest of the proof into 2 cases.
		
			{\it Case $\beta>1$ and  $p\ge N$}. 
			By the Sobolev embedding theorem,
			there exist a subsequence, still denoted by $\{u_n\}_{n\in\mathbb{N}},$ 
			and
			$u\in W^{1,p}_0(\Omega)$  such that
			\begin{equation}\label{eq:autovalor4}
				\begin{aligned}
					&u_n\rightharpoonup u \text{ weakly in }W^{1,p}_0(\Omega),\\
					&|u_n|^{\alpha}\to |u|^{\alpha} \text{ strongly in }
					L^{\frac{p}{\alpha}}(\Omega),\\
					&|u_n|^{\alpha}\to |u|^{\alpha} \text{ strongly in }
					L^{\frac{q}{q-\beta+1}}(\Omega).\\
				\end{aligned}
			\end{equation}
		
			On the other hand, by \eqref{eq:autovalor3} and 
			Lemma \ref{lema:Poincare1}, we have
			that $\{v_n\}_{n\in\mathbb{N}}$ is bounded in $W^{1,q}(\Omega).$ Hence,
			by the Sobolev embedding theorem,
			there exist a subsequence, still denoted by 
			$\{v_n\}_{n\in\mathbb{N}},$ and
			$v\in W^{1,q}(\Omega)$  such that
			\begin{equation}\label{eq:autovalor5}
				\begin{aligned}
					&v_n\rightharpoonup v \text{ weakly in }W^{1,q}(\Omega),\\
					&|v_n|^{\beta}\to |v|^{\beta} \text{ strongly in }
					L^{\frac{q}{\beta}}(\Omega),\\
					&|v_n|^{\beta-2}v_n\to |v|^{\beta-2}v \text{ strongly in }
					L^{\frac{q}{\beta-1}}(\Omega).\\
				\end{aligned}
			\end{equation}
		
			By  \eqref{eq:autovalor3}, \eqref{eq:autovalor4}, and
			\eqref{eq:autovalor5}, we have that
			\begin{equation}\label{eq:autovalor6}
				\lambda_{p,q}^{\alpha,\beta} \ge\int_{\Omega}\dfrac{|\nabla u|^p}{p}\, dx
				+\int_{\Omega}\dfrac{|\nabla v|^q}{q}\, dx.
			\end{equation}
		
			On the other hand, by \eqref{eq:autovalor1}, \eqref{eq:autovalor2.33},
			\eqref{eq:autovalor4}, and \eqref{eq:autovalor5}, we get
			\[	
				\int_{\Omega}|u|^{\alpha}|v|^{\beta-2}v\, dx=0,\text{ and }
				\int_{\Omega}|u|^{\alpha}|v|^{\beta}\, dx=1.
			\]
			Then $(u,v)\in\mathcal{A}_{p,q}^{\alpha,\beta},$ and by \eqref{eq:autovalor6} and
			\eqref{eq:autovalor} we have that
			\[
				\lambda_{p,q}^{\alpha,\beta} =\int_{\Omega}\dfrac{|\nabla u|^p}{p}\, dx
				+\int_{\Omega}\dfrac{|\nabla v|^q}{q}\, dx,
			\]
			that is $(u,v)$ is a minimizer of \eqref{eq:autovalor}. Therefore $(u,v)$
			 is an eigenfunction corresponding to $\lambda_{p,q}^{\alpha,\beta}.$
		
			{\it Case $q>N$}. By the Sobolev embedding theorem,
			there exist a subsequence, still denoted by $\{u_n\}_{n\in\mathbb{N}},$ and
			$u\in W^{1,p}_0(\Omega)$  such that
			\begin{equation}\label{eq:autovalor7}
				\begin{aligned}
					u_n\rightharpoonup u &\text{ weakly in }W^{1,p}_0(\Omega),\\
					|u_n|^{\alpha}\to |u|^{\alpha} &\text{ strongly in }
					L^{\frac{p}{\alpha}}(\Omega).
				\end{aligned}
			\end{equation}
		
			On the other hand, by \eqref{eq:autovalor3} and Lemma \ref{lema:Poincare2}, we have
			that $\{v_n\}_{n\in\mathbb{N}}$ is bounded in $W^{1,q}(\Omega).$ Hence,
			by the Sobolev embedding theorem,
			there exist a subsequence, still denoted by $\{v_n\}_{n\in\mathbb{N}},$ and
			$v\in W^{1,q}(\Omega)$  such that
			\begin{equation}\label{eq:autovalor8}
				\begin{aligned}
					v_n\rightharpoonup v &\text{ weakly in }W^{1,q}(\Omega),\\
					v_n\to v &\text{ strongly in } C(\overline{\Omega}).
				\end{aligned}
			\end{equation}
		
			By  \eqref{eq:autovalor3}, \eqref{eq:autovalor7}, and
			\eqref{eq:autovalor8}, we have that
				\begin{equation}\label{eq:autovalor9}
				\lambda_{p,q}^{\alpha,\beta} \ge\int_{\Omega}\dfrac{|\nabla u|^p}{p}\, dx
				+\int_{\Omega}\dfrac{|\nabla v|^q}{q}\, dx.
			\end{equation}
		
			On the other hand, by \eqref{eq:autovalor1}, \eqref{eq:autovalor2.33},
			\eqref{eq:autovalor7}, and \eqref{eq:autovalor8}, we get
			\[
				\int_{\Omega}|u|^{\alpha}|v|^{\beta-2}v\, dx=0 \qquad \text{ and } \qquad
				\int_{\Omega}|u|^{\alpha}|v|^{\beta}\, dx=1,
			\]
			since $\beta>1.$ Then $(u,v)\in\mathcal{A}_{p,q}^{\alpha,\beta},$ and by
			\eqref{eq:autovalor9} and \eqref{eq:autovalor} we have that
			\[
				\lambda_{p,q}^{\alpha,\beta} =\int_{\Omega}\dfrac{|\nabla u|^p}{p}\, dx
				+\int_{\Omega}\dfrac{|\nabla v|^q}{q}\, dx,
			\]
			which concludes the proof.
    	\end{proof}

    \begin{remark}
		Note that, if $(u,v)$ is a minimizer of \eqref{eq:autovalor} then so is 
		$(|u|,v),$
		that is if  $(u,v)$ is a solution of 
		\eqref{eq:problema}--\eqref{eq:bc} with
		$\lambda=\lambda_{p,q}^{\alpha,\beta} $ then we can assume 
		that $u\ge0.$ Moreover, due to the results in \cite{JLV}, we get 
		$u>0$ in 
		$\Omega$ for $p$ and $q$ large enough.
	\end{remark}

%%%%%%%%%%%%%%%%%%%%%%%%%%%%%%%%%%%%%%%%%%%%%%%%%%%%%%%%%%%%%%%%%%%%%%%%%%%%%%%%%%%%%%%%%%
\section{The limit as $p,q\to \infty$} \label{sect-lim}
%%%%%%%%%%%%%%%%%%%%%%%%%%%%%%%%%%%%%%%%%%%%%%%%%%%%%%%%%%%%%%%%%%%%%%%%%%%%%%%%%%%%%%%%%%
	From now on, to simplify the notation, we write
	$\lambda_{p,q}$ instead of $\lambda_{p,q}^{\alpha,\beta}$
	and by $(u_{p,q},v_{p,q})$ we denote an eigenfunction corresponding to 
	$\lambda=\lambda_{p,q}^{\alpha,\beta} $
 normalized with $\int_{\Omega} |u_{p,q}|^\alpha|v_{p,q}|^{\beta}\, dx =1$.
	
	Recall that we have assumed that
	$$
		\frac{\alpha}{p} \to \Gamma \in (0,1), \quad\mbox{and}\quad \frac{q}{p} 
		\to Q\in(0,\infty)
		\quad\mbox{as } p,q\to \infty.
	$$
    In addition, since $\frac{\alpha}p + \frac{\beta}{q} =1,$ 
    we get
	$$
		\frac{\beta}{q} \to 1-\Gamma \quad\mbox{as } p,q\to \infty.
	$$
	
	Now, we deal with the limit as $p,q\to \infty$ in a variational setting 
	(showing the first part of Theorem \ref{teo:autovalor.limite.intro}).

	\begin{lema} \label{lema.con.unif} 
		Under the assumption \eqref{eq:A},
 		there exists a sequence $\{(p_n,q_n)\}_{n\in\mathbb{N}}$ such that
 		$p_n,q_n \to \infty,$
		$$
			 u_{n} \to u_\infty,\quad v_{n} \to v_\infty
			 \quad \mbox{ uniformly in } \overline{\Omega}
			 \mbox{ as } n\to \infty, 
		$$
		where $(u_n,v_n)$ is an eigenfunction
		corresponding to $\lambda_{p_n,q_n}$ for all $n\in \mathbb{N}.$
		Moreover, 
		$$
			\left(\lambda_{p,q}\right)^{1/p} 
			\to \Lambda_\infty(\Gamma,Q) \coloneqq
			\inf \left\{
			\frac{\displaystyle \max \left\{ 
			\| \nabla w \|_{L^\infty (\Omega)}; \,  
			\| \nabla z \|^Q_{L^\infty (\Omega)}\right\} }{
			\displaystyle \| |w|^{\Gamma} 
			|z|^{(1-\Gamma)Q} \|_{L^\infty (\Omega)}}
			\colon (w,z)\in\mathcal{A_{\infty}}\right\}
 		$$
		as $p,q\to \infty$ and $(u_\infty,v_\infty)$ is a minimizer of
		$\Lambda(\Gamma,Q)$.
	\end{lema}

	\begin{proof}
		We first look for a uniform bound for $\left(\lambda_{p,q}\right)^{1/p}$.
		To this end, let us consider a non-negative Lipschitz function 
		$w\in W^{1,\infty }(\Omega)$ that vanishes on $\partial \Omega$.
		
		Once this functions is fixed we choose $z\in W^{1,\infty } (\Omega)$ a Lipschitz 
		function and after that we choose $K=K(p,q)$ such that
		$$
			\int_{\Omega}|w|^{\alpha}|(z-K)|^{\beta-2}(z-K)
             \, dx=0.
		$$
		Note that $K(p,q)$ is bounded, in fact, we have $\inf\{z(x)\colon x\in \Omega\} 
		\leq K(p,q) \leq \sup\{z(x)\colon x\in\Omega\}$.
		We normalize according to
		$$
			\int_{\Omega}|w|^{\alpha}|(z-K)|^{\beta}
             \, dx=1.
		$$

		Hence, using the pair $(w,z-K)$ as test in \eqref{eq:autovalor2.33} we get
		$$
			\lambda_{p,q} \leq \displaystyle\int_{\Omega}\dfrac{|\nabla w|^p}{p}\, dx +
       		 \int_{\Omega}\dfrac{|\nabla z|^q}{q}\, dx.
		$$
		Therefore
		\begin{equation}\label{cota.lambda}
			\begin{aligned}
				\displaystyle
					\limsup_{p,q \to \infty } (\lambda_{p,q})^{\nicefrac1p}
						&\leq \limsup_{p \to \infty }  
						\left\{ \frac1p \| \nabla z \|^p_{L^p (\Omega) }  +
						\frac1q \| \nabla w \|^q_{L^q (\Omega)  } \right\}^{1/p} \\[3pt]
                      & =  \max \left\{ \| \nabla z \|_{L^\infty (\Omega)  }  ;
                       \| \nabla w \|^Q_{L^\infty (\Omega)  } \right\}
					\leq C.
			\end{aligned}
		\end{equation}
		Thus, 
		there is a constant, $C$, independent of $p$ and $q$ such that, 
		for $p$ and $q$ large,
		$$ 
			(\lambda_{p,q})^{1/p} \leq C. 
		$$

		Let $(u_{p,q},v_{p,q})$ be a minimizer for $\lambda_{p,q}$ normalized by
		$
			\int_{\Omega}|u_{p,q}|^{\alpha}|v_{p,q}|^{\beta}
             \, dx=1.
		$
		Then, we have that
		$$ 
			\frac1p  \int_\Omega |\nabla u_{p,q}|^p \, dx + 
			\frac1q\int_\Omega | \nabla v_{p,q}|^q \, dx =  \lambda_{p,q},  $$
		from which we deduce using \eqref{cota.lambda} that
		\begin{equation}\label{cota.norma}
			\begin{aligned}
				\limsup_{p,q\to \infty} \| \nabla u_{p,q} \|_{L^p (\Omega)} &\le  
				\limsup_{p,q\to \infty}\left(p\lambda_{p,q}\right)^{\nicefrac1p} 
				=\limsup_{p,q\to \infty}\left(\lambda_{p,q}\right)^{\nicefrac1p} 
				\leq C,\\[5pt]
				\limsup_{p,q\to \infty} \| \nabla v_{p,q} \|_{L^q (\Omega)} &\le
  				\limsup_{p,q\to \infty}\left(q\lambda_{p,q}\right)^{\nicefrac1q} 
  				=\limsup_{p,q\to \infty}
  				\left[\left(\lambda_{p,q}\right)^{\nicefrac1p}\right]^{
  				\nicefrac{p}{q}}\\&=
  				\left[\limsup_{p,q\to \infty}
  				\left(\lambda_{p,q}\right)^{\nicefrac1p}\right]^{
  				\nicefrac{1}{Q}}   \leq C.
			\end{aligned}	
		\end{equation}

		Now, we argue as follows: We fix $r\in (N,\infty)$. 
		Using Holder's inequality, we obtain for $p,q>r$ large enough that
		\begin{equation}\label{cota.norma2}
			\left(\int_\Omega | \nabla u_{p,q}|^r  \,dx \right)^{\nicefrac1r} \le
			\left(\int_\Omega | \nabla u_{p,q}|^p  \,dx \right)^{\nicefrac1p} 
			|\Omega|^{\frac1r - \frac1p } \leq C. 
		\end{equation}
		Analogously, we have
		$$ 
			\left(\int_\Omega | \nabla v_{p,q}|^r \,dx  \right)^{\nicefrac1r} 
			\le
			\left(\int_\Omega | \nabla v_{p,q}|^q \,dx  \right)^{\nicefrac1q} 
			|\Omega|^{\frac1r - \frac1q } \leq C. 
		$$
		Hence, extracting a subsequence $\{(p_n,q_n)\}_{n\in\mathbb{N}}$ 
		$p_n,q_n\to \infty$ if necessary, we have that
		$$ 
			u_{n}=u_{p_n,q_n} \rightharpoonup u_\infty  
			\quad\text{and}\quad
			v_{n}=v_{p_n,q_n} \rightharpoonup v_\infty  
		$$
		weakly in $W^{1,r} (\Omega)$ for any $N<r<\infty$ 
		and uniformly in $\overline{\Omega}$.

		From \eqref{cota.norma} and \eqref{cota.norma2}, 
		we obtain that this weak limit verifies
		$$ 
			\left( \int_{\Omega} | \nabla u_{\infty} |^r \, dx \right)^{\nicefrac1r} 
			\le  |\Omega|^{\nicefrac1r}  
			\limsup_{p,q\to \infty}(\lambda_{p,q})^{\nicefrac1p}. 
		$$
		As we can assume that the above inequality 
		holds for every $r>N$ (using a diagonal argument), we get 
		that $u_\infty \in W^{1,\infty}_0 (\Omega)$ and moreover, taking the
		limit as $ r \to \infty $, we obtain
		$$ 
			| \nabla u_{\infty}(x) | \le  
			\liminf_{p,q\to \infty}(p\lambda_{p,q})^{\nicefrac1p}=
			\liminf_{p,q\to \infty}(\lambda_{p,q})^{\nicefrac1p} 
			\qquad \hbox{  a.e. }x \in \Omega. 			
		$$

		Analogously, we obtain that the function $v_\infty$ verifies that 
		$v_\infty\in W^{1,\infty}(\Omega)$ and
		\begin{align*}
			| \nabla v_{\infty}(x)| &\le  \liminf_{p,q\to \infty}
			\left(q\lambda_{p,q}\right)^{\nicefrac1q}=\liminf_{p,q\to \infty}
			\left(\lambda_{p,q}\right)^{\nicefrac1q}
			=\liminf_{p,q\to \infty}\left(\lambda_{p,q}\right)^{\nicefrac1q}\\
			&=\liminf_{p,q\to \infty}\left[\left(\lambda_{p,q}\right)^{\nicefrac1p}
			\right]^{\nicefrac{p}q}
			=\left[\liminf_{p,q\to \infty}\left(\lambda_{p,q}\right)^{\nicefrac1p}
			\right]^{\nicefrac{1}Q}
			 \qquad \hbox{  a.e. }x \in \Omega,  
		\end{align*}
		Then
		\[
			| \nabla v_{\infty}(x)|^Q\le
			\liminf_{p,q\to \infty}\left(\lambda_{p,q}\right)^{\nicefrac1p}
			\qquad \hbox{  a.e. }x \in \Omega,  
		\]

		From the uniform convergence and the normalization condition, we obtain that
		\begin{equation}\label{eq:inf1}
			\| |u_\infty|^\Gamma |v_\infty|^{(1-\Gamma)Q} \|_{L^\infty (\Omega)} =1,  
		\end{equation}
		and from
		$$
			\int_{\Omega}|u_{p,q}|^{\alpha}|v_{p,q}|^{\beta-2}v_{p,q}\, dx=0,
		$$
		we get
		\begin{equation}\label{eq:inf2}
				\max_{x\in\Omega} |u_\infty(x)|^\Gamma |(v_\infty(x))_+|^{(1-\Gamma)Q}  
			=  \max_{x\in\Omega} |u_\infty(x)|^\Gamma |(v_\infty(x))_-|^{(1-\Gamma)Q}.  
		\end{equation}
		Therefore, $(u_\infty,v_\infty)\in\mathcal{A}_{\infty}$ and we get
		\begin{equation}\label{eq:des1}
			\Lambda_\infty(\Gamma,Q)\le\frac{\displaystyle \max 
			\left\{ \| \nabla u_\infty\|_{L^\infty (\Omega)}; \,  
			\| \nabla v_\infty\|^Q_{L^\infty (\Omega)}\right\}}
			{\displaystyle \| |u_\infty|^{\Gamma} 
			|v_\infty|^{(1-\Gamma)Q} \|_{L^\infty (\Omega)} } \leq 
			\liminf_{p,q\to \infty}(\lambda_{p,q})^{\nicefrac1p}.
		\end{equation}

		Now, we note that since $K(p,q)$ is bounded, 
		there is a sequence $\{(p_n,q_n)\}$ such that 
		$$
			p_n, q_n\to \infty \qquad \mbox{and} \qquad K(p_n,q_n) \to k
		$$
		as $n\to \infty$. From \eqref{cota.lambda}, we get
		\begin{equation}\label{eq:des2}
			\limsup_{p,q\to \infty}(\lambda_{p,q})^{\nicefrac1p} \leq
			\frac{\displaystyle \max \left\{ \| \nabla w \|_{L^\infty (\Omega)}; 
			\,  \| \nabla (z-k) \|^Q_{L^\infty (\Omega)}
			\right\} }{ \displaystyle \| |w|^{\Gamma} |(z-k)|^{(1-\Gamma)Q} 
			\|_{L^\infty (\Omega)} }
		\end{equation}
		for every pair $(w,z-k)$ with
		$$ 
			\max_{x\in\Omega} |w(x)|^\Gamma |(z(x)-k)_+|^{(1-\Gamma)Q}  =  
			\max_{x\in\Omega} |w(x)|^\Gamma |(z(x)-k)_-|^{(1-\Gamma)Q} .  
		$$
		Thus
		\begin{equation}\label{eq:des3}
			\limsup_{p,q\to \infty}(\lambda_{p,q})^{\nicefrac1p} \leq
			\Lambda_{\infty}(\Gamma,Q).
		\end{equation}
		Therefore, by \eqref{eq:des1} and \eqref{eq:des3}, 
		we get
		$$ 
			\left(\lambda_{p,q}\right)^{\nicefrac1p}\to \Lambda_\infty(\Gamma,Q)
 		$$
 		as $p,q\to\infty,$ and $(u_\infty,v_\infty)$ is a minimizer of
 		$\Lambda_{\infty}(\Gamma,Q).$
 	\end{proof}

%%%%%%%%%%%%%%%%%%%%%%%%%%%%%%%%%%%%%%%%%%%%%%%%%%%%%%%%%%%%%%%%%%%%%%%%%%%%%%%%%%%%%%%%%%%%%%%
\section{The value of $\Lambda_\infty$ in a ball and in a rectangle.} 
\label{sect-ex}
%%%%%%%%%%%%%%%%%%%%%%%%%%%%%%%%%%%%%%%%%%%%%%%%%%%%%%%%%%%%%%%%%%%%%%%%%%%%%%%%%%%%%%%%%%%%%%
	\subsection{The case of a ball.} Now our aim is to compute the limit value 
	$\Lambda_\infty$ in the ball of radius $R$, that we denote as $B_R$.

	By symmetry reasons we have to choose $x_0=(a,0,\dots,0)$ with $0<a<R$, 
	the point where
	$$
		\| |u_\infty|^{\Gamma} |v_\infty|^{(1-\Gamma)Q} \|_{L^\infty (B_R)}
		= |u_\infty|^{\Gamma} |v_\infty|^{(1-\Gamma)Q}(x_0)=1.
	$$
	Note that we can choose $v_\infty$ to be symmetric (odd in the $x_1$-direction), 
	that is, $v_\infty (x_1,x_2,\dots,x_N)= - v_\infty (- x_1,x_2,\dots,x_N)$.

	Now we are lead to compute:
	$$
		\max \left\{ \| \nabla u_\infty\|_{L^\infty (B_R)}; \,  
		\| \nabla v_\infty\|^Q_{L^\infty (B_R)}\right\}.
	$$
	Observe that the best choice that we can make is to take $u_\infty$ as the cone
	$$
			u_\infty (x) = k_1 (R-|x|).
	$$
	Then we have
	$$
		\| \nabla u_\infty\|_{L^\infty (B_R)} = k_1
	\qquad \mbox{and} \qquad
		u_\infty (x_0) = k_1 (R-a).
	$$
	Concerning $v_\infty$ we can choose a plane
	$$
		v_\infty (x) = k_2 \langle x, e_1 \rangle.
	$$
	Then we have
	$$
		\| \nabla v_\infty\|_{L^\infty (B_R)} = k_2
		\qquad\mbox{and}\qquad
		v_\infty (x_0) = k_2 a.
	$$
	These functions $u_\infty$ and $v_\infty$ are depicted in the following figure.
	
	
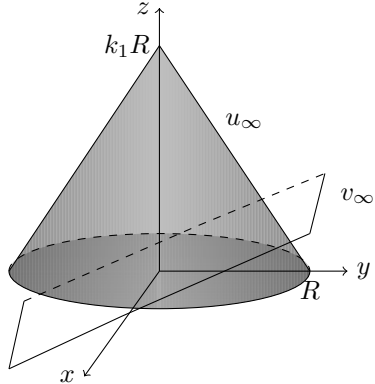
\begin{center}
		\begin{tikzpicture}
		\fill[
  			top color=gray!50,
 			bottom color=gray!10,
  			shading=axis,
  			opacity=0.25
  			] 
  			(0,0) circle (2cm and 0.5cm);
			\fill[
 				left color=gray!50!black,
  				right color=gray!50!black,
  				middle color=gray!50,
  				shading=axis,
  				opacity=0.25
  				] 
  			(2,0) -- (0,3) -- (-2,0) arc (180:360:2cm and 0.5cm);
			\draw 
  				(-2,0) arc (180:360:2cm and 0.5cm) -- (0,3) -- cycle;
			\draw[dashed]
  				(-2,0) arc (180:0:2cm and 0.5cm);
  			\draw [->]
  				(0,0) -- (2,0) node[below] {$R$} -- (2.5,0)
  				node[right] {$y$} ;
			\draw[->]
  				(0,0) --  (0,3) node[left] {$k_1R$}-- (0,3.5)
  				node[left] {$z$};
  			\draw[->]
  				(0,0) -- (-1,-1.4) node[left] {$x$};  
  			\draw
  				(-2,-1.3) -- (-.5,-.6) -- (2,.5) ;
  			\draw[dashed]
  				(-1.8,-0.4) -- (.5,.6) -- (2.2,1.3) ;
  			\draw
  				(-1.8,-0.4) -- (-2,-1.3);
  			\draw
  				(2,.5) -- (2.2,1.3) ;
  			\draw
  				(1.5,2) node[left] {$u_{\infty}$}  ;
  			\draw
  				(3,1) node[left] {$v_{\infty}$}  ;
		\end{tikzpicture}
	\end{center}

	Now we have to compute	
	$$
		\min_{k_1,k_2,a} \max \left\{ k_1; \,  k_2^Q\right\}
	$$
	with the restriction
	\begin{align*}
		\max_{x\in B_R}  |u_\infty|^{\Gamma} |v_\infty|^{(1-\Gamma)Q}
		& = \max_{0\leq s \leq R} ( k_1 (R-s))^{\Gamma} 
		(k_2 s)^{(1-\Gamma)Q}\\
		&=
		k_1^{\Gamma} k_2^{(1-\Gamma)Q}(R-a)^{\Gamma}a^{(1-\Gamma)Q}=1.
	\end{align*}
	Then we have to compute
	$$
		\max_{0\leq s \leq R} (R-s)^{\Gamma} s^{(1-\Gamma)Q}.
	$$
	We have that this maximum is attained at a point $a$ that satisfies 
	$$
		\Gamma a = Q(1-\Gamma) (R-a),
	$$
	hence, $a$ is given by
	$$
	a = \dfrac{Q(1-\Gamma)R}{\Gamma+Q(1-\Gamma)}.
	$$
	Therefore, the restriction is given by
	$$
		k_1^{\Gamma} k_2^{(1-\Gamma)Q} \left(
		\dfrac{\Gamma R}{\Gamma+Q(1-\Gamma)}\right)^{\Gamma}  
		\left(\dfrac{Q(1-\Gamma) R}{\Gamma+Q(1-\Gamma)}\right)^{(1-\Gamma)Q}=1.
	$$
	This gives
	$$
		k_1 = \Theta k_2^{\frac{\Gamma-1}{\Gamma}Q}
	$$
	with
	$$
		\Theta =\dfrac{\Gamma+Q(1-\Gamma)}{\Gamma R} 
		\left( \frac{\Gamma+Q(1-\Gamma)}{Q(1-\Gamma)R} \right)^{\frac{(1-\Gamma)Q}
		\Gamma}.
	$$
	Finally we arrive to
	$$
		\min_{k_2} \max 
		\left\{  \Theta k_2^{\frac{\Gamma-1}{\Gamma}Q}; \,  k_2^Q
		\right\}.
	$$
	We must have
	$$
		\Theta k_2^{\frac{\Gamma-1}{\Gamma}Q} = k_2^Q,
	$$
	and hence
	$$
		k_2 = \Theta^{\frac{\Gamma}Q}.
	$$
	We conclude that the optimal value for $\Lambda_\infty$ is given by
	$$
		\Lambda_\infty(\Gamma,Q) =  \Theta^{\Gamma} =
			\left(\dfrac{\Gamma+Q(1-\Gamma)}{\Gamma R}
			\right)^{\Gamma} 
		\left( \frac{\Gamma+Q(1-\Gamma)}{Q(1-\Gamma)R} 
		\right)^{(1-\Gamma)Q}.
	$$

	\subsection{The case of a rectangle.} Now we want to compute 
	$\Lambda_\infty(\Gamma,Q)$ 
	when $\Omega$ is the rectangle  $(-R,R)\times (-L,L) \subset \R^2$.
	Without loss of generality, we assume that $L\leq R.$

	Here, as for the case of the ball, we rely on symmetry. 
	We look for a point $x_0=(a,0)$ with $L\le a<R$, where
	$$
		\| |u_\infty|^{\Gamma} |v_\infty|^{(1-\Gamma)Q} \|_{L^\infty (\Omega)}
		= |u_\infty|^{\Gamma} |v_\infty|^{(1-\Gamma)Q}(x_0)=1.
	$$
	Note that we can choose $v_\infty$ to be symmetric (odd in the 
	$x$-direction), that is 
	$v_\infty (x,y)=- v_\infty (- x, y)$.

	Observe that the best choice that we can make is to take 
	$u_\infty$ as the cone
	$$
		u_\infty (x) = k_1 (\rho-|(x,y)-(a,0)|)_+,
	$$
	with $\rho=R-a \leq\min\{L,R-L\}.$ Then we have
	$$
		\| \nabla u_\infty\|_{L^\infty (\Omega)} = k_1 \qquad \mbox{and} \qquad
		u_\infty (x_0) = k_1 \rho .
	$$
	
	Concerning $v_\infty$, as before, we can choose a plane
	$$
		v_\infty (x) = k_2 x_1.
	$$
	Then we have
	$$
		\| \nabla v_\infty\|_{L^\infty (B_R)} = k_2 \qquad 
	\mbox{and} \qquad
		v_\infty (x_0) = k_2 a.
	$$

	Now we have to compute
	$$
		\min_{k_1,k_2,a} \max \left\{ k_1; \,  k_2^Q
		\right\}
	$$
	with the restriction
	\begin{align*}
		\max_{x\in\Omega}  |u_\infty|^{\Gamma} |v_\infty|^{(1-\Gamma)Q} 
		&= \max_{a\leq s \leq R} ( k_1 (R-s))^{\Gamma} (k_2 s)^{(1-\Gamma)Q}\\
		&=k_1^{\Gamma} k_2^{(1-\Gamma)Q} (R-a)^{\Gamma}  a^{(1-\Gamma)Q} =1.
	\end{align*}

	Then we have to compute
	\begin{equation}\label{eq:rectangulo}
		\max_{a\leq s \leq R} (R-s)^{\Gamma} s^{(1-\Gamma)Q}.
	\end{equation}
	
	When $\rho<L,$ this maximum is attained at a point $a$ that is given by
	$$
		\Gamma a = Q(1-\Gamma) (\rho+a-a),
	$$
	that is
	$$
		a = Q\frac{(1-\Gamma)}{\Gamma} \rho.
	$$

	Hence, with similar computations as the ones that we did for the ball we 
	obtain that  
	$$
		\Lambda_\infty(\Gamma,Q)=\left(\dfrac{\Gamma+Q(1-\Gamma)}{\Gamma R}
						\right)^{\Gamma} 
						\left( \dfrac{\Gamma+Q(1-\Gamma)}{Q(1-\Gamma)R} 
						\right)^{(1-\Gamma)Q}
						\qquad  \mbox{ if } 
						\dfrac{\Gamma R}{Q(1-\Gamma)} \leq L. 
	$$
	Observe that, in this case, $\Lambda _\infty(\Gamma,Q)$ coincides with the 	eigenvalue that we found in the case of the ball.
	
	When $\rho=L,$ \eqref{eq:rectangulo} 
	is attained at a point $a$ that is given by
	$a=R-L$ then
	$$
		\Lambda_\infty(\Gamma,Q)=\frac{1}{(R-L)^{\Gamma}  L^{1-\Gamma}  } 
						\qquad  \mbox{ if } 
						\dfrac{\Gamma R}{Q(1-\Gamma)} > L. 
	$$
	
	Note that computing the value of $\Lambda_\infty(\Gamma,Q)$ 
	for a general domain $\Omega$ is not straightforward.

%%%%%%%%%%%%%%%%%%%%%%%%%%%%%%%%%%%%%%%%%%%%%%%%%%

\section{Viscosity solutions} \label{sect-viscosity}

	In order to identify the limit PDE problem satisfied by 
	any limit $(u_\infty,v_\infty)$, we introduce the definition 
	of viscosity solutions. Since we deal with different boundary 
	conditions for the components $u_{p,q}$ (Dirichlet) and $v_{p,q}$ 
	(Neumann) we split the passage to the limit into two parts. 
	Also remark that $u_\infty$ is non-negative in $\overline{\Omega}$ 
	but $v_\infty$ changes sign. This is reflected in the fact that 
	they are solutions to quite different equations. First, we 
	deal with the equation and boundary condition verified by 
	$u_\infty$ and next we deal with $v_\infty$.

	\subsection{Passing to the limit in $u_{p,q}$.}
		Assuming that $u_{p,q}$ is smooth enough, we can rewrite the first 
		equation in \eqref{eq:problema} as
 		\begin{equation}\label{superecua}
 		\begin{array}{ll}
 			-|\nabla u_{p,q}|^{p-4} \left(|\nabla u_{p,q}|^{2}\Delta u_{p,q}
 			+(p-2) \Delta_\infty u_{p,q}
 			\right)=\alpha \lambda_{p,q}u_{p,q}^{\alpha-1}v_{p,q}^{\beta}.
			 \end{array}
 		\end{equation}
		Recall that $-\Delta_{\infty} u = -  
		\nabla u D^2 u (\nabla u)^t$.
		This equation is non-linear, elliptic (degenerate) 
		but not in divergence form, 
		thus it makes sense to consider viscosity sub-solutions and
		super-solutions of it.
		Let $x\in\Omega,$ $y\in \R$, $ z\in \R^N$, and $
		S$ a real symmetric matrix. 
		We consider the following function
		\begin{equation}\label{superoper}
 			H_{p}(x,y,z,S) =  \displaystyle 
 			-|z|^{p-4}\left(|z|^{2}\mbox{trace}(S) + (p-2) \langle S	
 			\cdot z,z\rangle\right) \\[8pt]
            \displaystyle - \alpha \lambda_{p,q} 
            |y|^{\alpha-2}y v_{p,q}(x)^\beta.
 		\end{equation}
		Observe that $H_p$ is elliptic in the sense that 
		$H_p(x,y,z,S)\ge H_p(x,y,z,S')$ if 
		$S\le S'$ in the sense of bilinear forms, and also that
		\eqref{superecua} can be written as 
		$H_{p}(x, u_{p,q},\nabla u_{p,q},D^2 u_{p,q})=0$.
 		We are thus interested in viscosity super and 
 		sub solutions of the partial differential equation
 		\begin{equation}\label{visco}
 			\begin{cases}
					H_{p}(x, u,\nabla u,D^2 u)=0 
 					 & \mbox{ in }\Omega,\\
 					u = 0  & \mbox{ on }\partial \Omega.
			\end{cases}
 		\end{equation}

 		\begin{defi} 
 			An upper semi-continuous function $u$ defined in 
 			$\Omega$ is a {\it viscosity sub-solution} of 
 			\eqref{visco} if, $u|_{\partial\Omega} \leq 0$ 
 			and, whenever $x_0\in\Omega $ and
			$\phi\in C^2(\Omega)$ are such that
			$u(x_0)=\phi(x_0)$ and $u-\phi$ has a strict local maximum 
			point at  $x_0$,
 			then
 			\[
 				H_{p}(x_0,\phi(x_0),\nabla\phi (x_0),D^2\phi(x_0))\leq 
 				0.
 			\]
 		\end{defi}

		\begin{defi} 
			A lower semi-continuous
			function $u$ defined in $\Omega$ is a 
			{\it viscosity super-solution} of 
			\eqref{visco} if, $u|_{\partial\Omega} \geq 0$
			and, whenever $x_0\in\Omega$ and 
			$\phi\in C^2(\Omega)$ are such that
			$u(x_0)=\phi(x_0)$ and $u-\phi$ has a strict local minimum 
			point at  $x_0$,
 			then
 			\[
 				H_{p}(x_0, \phi(x_0),
 				\nabla\phi(x_0),D^2\phi(x_0))\geq 0.
 			\]
 		\end{defi}

		We observe that in both of the above definitions 
		the second condition is required just in a neighbourhood of 
		$x_0$ and the strict inequality can be relaxed. We refer to
		\cite{CIL} for more details about general theory 
		of viscosity solutions, and to \cite{JLM2} 
		for viscosity solutions related to the 
		$\infty-$Laplacian and the $p-$Laplacian operators.
		The following result can be shown as in 
		\cite[Proposition 2.4]{MRU}, therefore we omit the proof here.

		\begin{lema}\label{sol.debil.es.viscosa}
			A continuous weak solution to the equation
			\begin{equation}\label{1.1.u}
				\begin{cases}
					-\Delta_p u  = \lambda \alpha |
					u|^{\alpha-2} u v_{p,q}^\beta  & \text{in } 
					\Omega, \\
					u=0  & \text{on } \partial \Omega,
				\end{cases}
			\end{equation}
			is a viscosity solution to $(\ref{visco})$.
		\end{lema}
		
		Now, we have all the ingredients to compute the limit of 
		\eqref{visco} as $p\to\infty$ in the viscosity sense, that is, 
		to identify the limit equation
		verified by any uniform limit of $u_{p,q}$, $u_\infty$.
		For $x\in\Omega,$ $y\in \R,$ $z\in \R^N$ and 
		$S$ a symmetric real matrix, we 
		define the limit operator $H_\infty$ by
		\begin{equation}\label{infiope}
			H_{\infty}(x,y,z,S)=\min\{-\langle S\cdot z,z\rangle, 
			|z|-\Lambda_\infty(\Gamma,Q) |y|^{\Gamma} 
			|v_\infty|^{(1-\Gamma)Q} (x)\}.
		\end{equation}
		Note that $H_{\infty}(x,u,\nabla u,D^2 u)=0$  is the limit 
		equation that we are looking for.

		\begin{teo}
			A function $u_\infty$ obtained as a limit of a 
			subsequence of $\{u_{p,q}\}$ is a viscosity solution to 
			the problem
			\begin{equation}\label{LimitEq}
				\begin{cases}
						H_{\infty}(x,u,\nabla u ,D^2 u)=0
						& \mbox{in } \Omega, \\
							u=0 & \mbox{on } \partial \Omega,
				\end{cases}
			\end{equation}
			with $H_\infty$ defined  in $(\ref{infiope})$, and 
			$v_\infty$ a uniform limit of $v_{p,q}$. 
		\end{teo}

		\begin{proof}
			In this proof we use ideas from \cite{BoRoSa}. 
			We consider a subsequence 
			$\{(p_{n},q_n)\}_{n\in\mathbb{N}}$ such that
			$p_n,q_n\to\infty$
			$$
				\lim_{n\to\infty}u_{p_n,q_n}=u_\infty,\qquad 
				\lim_{n\to\infty}v_{p_n,q_n}=v_\infty
			$$
			uniformly in $\Omega$ and 
			$(\lambda_{p_n,q_n})^{\nicefrac1p_n}\to
			\Lambda_\infty(\Gamma,Q)$. In what follows we omit 
			the subscript $n$ and denote as $u_{p,q}$, $v_{p,q}$ 
			and $\lambda_{p,q}$ such subsequences for simplicity.

			We first check that $u_\infty$ is a super-solution of 
			\eqref{LimitEq}.To this end, we consider a point 
			$x_0\in \Omega$ and a function $\phi\in C^2(\Omega)$ 
			such that $u_\infty(x_0)= \phi(x_0)$ and 
			$u_\infty(x)>\phi(x)$ for every $x\in B(x_0,R)$, 
			$x\neq x_0$, with $R>0$ fixed and verifying that 
			$B(x_0,2R)\subset\Omega.$ We must show that
			\begin{equation}\label{ToShow1}
				H_\infty(x_0,\phi(x_0),
				\nabla\phi(x_0),D^2\phi(x_0))\ge 0.
			\end{equation}

			Let $x_{p,q}$ be a minimum point of $u_{p,q}-\phi$ in 
			$\bar B(x_0,R)$. Since $u_{p,q}\to u_\infty$ 
			uniformly in $\bar{B}(x_0,R)$, up to a subsequence 
			$x_{p,q}\to x_0$.

			In view of Lemma \ref{sol.debil.es.viscosa}, 
			$u_{p,q}$ is a viscosity super-solution of 
			\eqref{visco}, then
				\begin{equation}\label{ppuff}
					\begin{aligned}
 						&-|\nabla \phi(x_{p,q})|^{p-4}
 						\Big(|\nabla \phi(x_{p,q}) |^{2}\Delta 
 						\phi(x_{p,q}) +(p-2)
 							\Delta_\infty \phi(x_{p,q}) \Big) \\
						& \qquad \geq  \alpha \lambda_{p,q}
						|\phi(x_{p,q})|^{\alpha-2}\phi(x_p)
						|v_{p,q}|^\beta(x_{p,q}). 
			\end{aligned}
			\end{equation}

			Assume that $\phi(x_0)=u_\infty (x_0)>0$ and 
			$|v_\infty| (x_0)>0$. Then for $p,q$ large, 
			$\phi(x_{p,q})>0$ and $|v_{p,q}|(x_{p,q})>0$ 
			so that the right hand side of \eqref{ppuff} is positive.
			It follows that $|\nabla\phi(x_{p,q})|>0$ and then we get
			\begin{equation}\label{mons2}
				\begin{array}{l}
					\displaystyle - 
					\left(\frac{|\nabla \phi(x_{p,q})|^{2}\Delta 
					\phi(x_{p,q})}{(p-2)}
					+\Delta_\infty \phi(x_{p,q}) \right) \\[10pt]
					 \displaystyle \geq 
					\left(\frac{
					\alpha^{\frac{1}{p}}}{(p-2)^{\frac{1}{p}}} 
					(\lambda_{p,q})^{\frac{1}{p}}
					|\phi(x_{p,q})|^\frac{\alpha-2}{p}
					\phi^\frac1p (x_{p,q}) 
					|v_{p,q}|^{\frac{\beta}{p}}(x_{p,q}) |\nabla 
					\phi(x_{p,q})|^{-1+\frac{4}{p}}\right)^p.
				\end{array}
			\end{equation}
			Note that we have
			\begin{equation}\label{mons}
				\lim_{p,q\to \infty} - 
				\left(\frac{|\nabla \phi(x_{p,q})|^{2}
					\Delta \phi(x_{p,q})}{(p-2)}+\Delta_\infty 
				\phi(x_{p,q}) \right) = -\Delta_\infty \phi(x_0) < \infty.
			\end{equation}
			Hence
			\[
				\limsup_{p,q\to\infty}
				\frac{\alpha^{\frac{1}{p}}}{(p-2)^{\frac{1}{p}}} 
				(\lambda_{p,q})^{\frac{1}{p}} \phi^\frac{\alpha-1}{p} 
				(x_{p,q}) 
				|v_{p,q}|^{\frac{\beta}{p}}(x_p)
				|\nabla \phi(x_{p,q})|^{-1+\frac{4}{p}}\leq 1.
			\]
			Recalling that by assumption 
			$\frac{\alpha}{p}\to \Gamma$ and $\frac{q}{p}\to Q$ as 
			$p,q\to \infty$, we obtain
			\begin{equation}\label{maravillosisima}
  				\Lambda_{\infty}(\Gamma,Q) 
  				\phi^{\Gamma } (x_0) |v_\infty|^{(1-\Gamma)Q}(x_0) \leq 
  				|\nabla \phi(x_0)|
			\end{equation}
			and
			\begin{equation}\label{maravillosisima.2}
   		 		-\Delta_\infty \phi(x_0) \geq 0,
			\end{equation}
			which is \eqref{ToShow1}.

			Assume now that either $\phi(x_0)=u_\infty(x_0)=0$ or 
			$v_\infty (x_0)=0$. 
			In particular, \eqref{maravillosisima} holds.
			Note first that if $\nabla \phi(x_0)=0$ then 
			$\Delta_\infty\phi(x_0)=0$ 
			by definition so that \eqref{maravillosisima.2} holds.
			We now assume that $|\nabla\phi(x_0)|>0$ and write \eqref{mons2}. 
			The parenthesis in the right hand side goes to 0 as 
			$p,q\to \infty$ so that the right hand side goes to 0 and 
			\eqref{maravillosisima.2} follows.

			To complete the proof it just remains to see that $u_\infty$ is a
			viscosity sub-solution. Let us consider a point $x_0\in\Omega$ and
			a function $\phi\in C^2(\Omega)$ such that
			$u_\infty(x_0)=\phi(x_0)$ and $u_\infty(x)<\phi(x)$ for every $x$
			in a neighbourhood of $x_0$. We want to show that
			\[
				H_{\infty}(x_0,\phi(x_0),\nabla\phi(x_0),D^2\phi(x_0))\leq0.
			\]
			We first observe that if $\nabla\phi(x_0) =0$ the previous
			inequality trivially holds. Hence, let us assume that $\nabla
			\phi (x_0) \neq 0$. Now, we argue as follows: assuming that
			\begin{equation}\label{hypo}
				|\nabla\phi(x_0)|-\Lambda_\infty (\Gamma,Q)
				\phi^{\Gamma} (x_0) |v_\infty|^{(1-\Gamma)Q} (x_0)>0,
			\end{equation}
			we will show that
			\begin{equation}\label{conse}
				 -\Delta_\infty\phi(x_0)\leq0.
			\end{equation}
  			As before, using that $u_{p,q}$ is a viscosity sub-solution of 
  			\eqref{visco}, we get a sequence of points 
  			$x_{p,q}\to x_0$ such that
			\begin{equation}\label{amor}
				\begin{aligned}
 					&\displaystyle - 
 					\left(\frac{|\nabla \phi |^{2}\Delta 
 					\phi(x_{p,q})}{(p-2)}
					+\Delta_\infty \phi(x_{p,q}) \right) \\[10pt]
					&\qquad \displaystyle \leq 
					\left(\frac{\alpha^{1/p}}{(p-2)} 
					(\lambda_{p,q})^{1/p}|\phi|^{(\alpha-1)/p} (x_{p,q}) 
					|v_{p,q}|^{\beta/p}(x_{p,q})|\nabla \phi(x_{p,q})
					|^{-1+4/p}\right)^p.
				\end{aligned}
			\end{equation}
			Using \eqref{hypo} we get
			\[
				\limsup_{p,q\to\infty}
				\left(\frac{\alpha^{1/p}}{(p-2)} (\lambda_{p,q})^{1/p}
				|\phi|^{(\alpha-1)/p} (x_{p,q}) |v_{p,q}|^{\beta/p}(x_{p,q})|
				\nabla \phi(x_{p,q})|^{-1+4/p}\right)^p =0.
			\]
			Hence, we conclude \eqref{conse} taking limits in 
			\eqref{amor} and we obtain that
			\begin{equation}\label{forever.2}
				\min\{-\Delta_\infty\phi(x_0),\, |\nabla \phi (x_0)| 
				- \Lambda_\infty(\Gamma,Q) \phi^{\Gamma} (x_0) 
				|v_\infty|^{(1-\Gamma)Q}(x_0)\}
				\leq 0.
			\end{equation}

			The fact that $u_\infty=0$ on $\partial \Omega$ 
			is immediate from the uniform convergence of $u_{p,q}$ since
			$u_{p,q}=0$ on $\partial \Omega$.
		\end{proof}

	\subsection{Passing to the limit in $v_{p,q}$.}
		Let
			\begin{equation}\label{superoper.vv}
 				F_{q}(x,y,z,S) =  
 				\displaystyle -|z|^{q-4}\left(|z|^{2}\mbox{trace}(S) + (q-2) 
 				\langle S\cdot z,z\rangle\right)  
 				\displaystyle - \beta \lambda_{p,q} |u_{p,q}(x)|^{\alpha} 
 				|y|^{\beta-2}y.
 			\end{equation}
 		Now we deal with viscosity super and subsolutions of the partial 
 		differential equation
 		\begin{equation}\label{visco.vv}
 			\begin{cases}
					F_{q}(x, v,\nabla v,D^2 v)=0 & \mbox{ in }\Omega,\\
 					\dfrac{\partial v}{\partial \nu }=0
					 & \text{ on }\partial \Omega.
			\end{cases}
		\end{equation}

		Here, we have to pay special attention to the fact that 
		$v_{p,q}$ changes sign and to the boundary condition
		$\nicefrac{\partial v_{p,q}}{\partial \nu }=0$ on 
		$\partial \Omega$. To this end, following \cite{Bar}, 
		we introduce the following definition of viscosity solution for 
		the boundary value problem
		\begin{equation}\label{ec.viscosa.con.borde}
			\begin{cases}
					F_q (x,v,\nabla v, D^2 v )   = 0 & \mbox{in } \Omega,\\
					B(x,\nabla v)  = 0 & \mbox{on }
				\partial \Omega,
			\end{cases}
		\end{equation}
		where $B(x,z)=\langle z,\nu(x)\rangle.$

		\begin{defi} \label{def.sol.viscosa.1}
			A lower semi-continuous function $ u $ is a viscosity
			super-solution if for every 
			$ \phi \in C^2(\overline{\Omega})$ such
			that $ u-\phi $ has a local strict minimum at the point 
			$ x_0 \in\overline{\Omega}$ with $u(x_0)= \phi(x_0)$ we have: If
			$x_0\in\partial \Omega$  the inequality
			$$
				\max \{ 
				  F_q(x_0, \phi (x_0),\nabla\phi (x_0), D^2\phi (x_0)),
				 B(x_0,  \nabla \phi (x_0))  \} \ge 0
			$$ holds,
			and if $x_0 \in \Omega$ then we require
			$$
				F_q(x_0,\phi (x_0),
				 \nabla \phi (x_0), D^2\phi (x_0)) \ge 0.
			$$
		\end{defi}

		\begin{defi} \label{def.sol.viscosa.2}
			An upper semi-continuous function $u$ is a sub-solution if
			for every $ \phi  \in C^2(\overline{\Omega})$ such that 
			$ u-\phi $ has a local strict maximum at the point 
			$x_0 \in \overline{\Omega}$ with $u(x_0)= \phi(x_0)$ we have:  
			If $x_0\in \partial\Omega$ the inequality
			$$
				\min \{  F_q(x_0,\phi (x_0), \nabla
				\phi (x_0), D^2\phi (x_0)),
				B (x_0, 
				\nabla \phi (x_0)) \} \le 0
			$$ 
			holds, and if
			$x_0 \in \Omega$ then we require
			$$
				F_q(x_0,\phi (x_0),
				 \nabla \phi (x_0), D^2\phi (x_0)) \le 0 .
			$$
		\end{defi}

		As before, we have that  any continuous weak solution 
		of the second equation in \eqref{eq:problema} is a viscosity 
		solution of \eqref{ec.viscosa.con.borde}. 
		This fact can be proved as in \cite{GMPR,GAMPR,RoSaint}.

		We can now pass to the limit $p,q\to \infty$ to obtain 
		the equation satisfied by $v_\infty$.

		\begin{teo}
			A function $v_\infty$ obtained as a limit of a subsequence of
			$\{v_{p,q}\}$ is a viscosity solution of the equation
			\begin{equation}\label{LimitEq.vv}
				\begin{cases}
					F_{\infty}(x,v,\nabla v ,D^2 v)=0 
 					& \mbox{ in }\Omega,\\
 					\dfrac{\partial v}{\partial \nu }=0
 					 & \text{ on }\partial \Omega,
				\end{cases}
			\end{equation}
			with $F_\infty$ defined by 
			\begin{equation}
			\begin{aligned}
 				F_\infty &(x,y,z,S)\\
				&=\begin{cases}
				 	\min\,\{ -\langle S\cdot z,z\rangle, |z|-
				 	\Lambda_\infty(\Gamma,Q)^{\nicefrac{1}{Q}} 
				 	|u_\infty(x)|^{\nicefrac{\Gamma}{Q}}  |y|^{1-\Gamma}  \}
				 	 &\text{in } \{y>0\}, \\
 					\max\,\{ -\langle S\cdot z,z\rangle, -|z| +
 					\Lambda_\infty(\Gamma,Q)^{\nicefrac{1}{Q}} 
 					|u_\infty(x)|^{\nicefrac{\Gamma}{Q}}  |y|^{1-\Gamma}  \}  
 					 &\text{in } \{y<0\}, \\
 					 -\langle S\cdot z,z\rangle \quad & \text{in } \{y=0\}.
				\end{cases}
				\end{aligned}
			\end{equation}
		\end{teo}

		\begin{proof}
			We prove that $v_\infty$ is a super-solution of \eqref{LimitEq.vv}. 
			The proof of the fact that it is a sub-solution is similar.
			Fix some point $x_0\in\overline{\Omega}$ and a smooth function 
			$\phi$ such that $v_\infty-\phi$ has a local strict minimum at 
			$x_0$ with $v_\infty(x_0)=\phi(x_0)$. Since $v_{p,q}\to v_\infty$ 
			uniformly there exist $x_{p,q}\in \mbox{argmax}\,\{v_{p,q}-\phi\}$ 
			such that $x_{p,q}\to x_0$ as $p,q\to\infty$.

			Assume first that $x_0\in \Omega$, so that 
			$x_{p,q}\in \Omega$ for $p,q$ large.
			If $\nabla\phi (x_0)=0$ then we have $\Delta_\infty\phi(x_0)=0$. 
			We assume now that $\nabla \phi(x_0)\neq 0$.
			As $u_{p,q}$ is a viscosity solution of \eqref{visco.vv}, we  have
			\begin{equation}\label{Eq2Visc.vv}
 				F_q (x_p,v_{p,q}(x_{p,q}),\nabla\phi(x_{p,q}),
 				D^2\phi(x_p))\ge 0.
			\end{equation}

			Dividing this inequality by 
			$(q-2)|\nabla\phi(x_{p,q})|^{q-4}$ we obtain
			\begin{equation}\label{Eq3Visc}
				\begin{aligned}
					&-\Delta_\infty\phi(x_0)+o(1)\\
					&\ge  v_{p,q}(x_{p,q})|\nabla\phi|^2 (x_{p,q}) 
 				\left( \frac{\lambda_{p,q}^\frac{1}{q-2}
 				|u_{p,q}(x_{p,q})|^{\frac{\alpha}{q-2}}
 				|v_{p,q}(x_{p,q})|^{\frac{\beta-2}{q-2}}}{|\nabla
 				\phi(x_{p,q})|(q-2)^\frac{1}{q-2}}\right)^{q-2}.
				\end{aligned}
			\end{equation}

			If  $v_\infty(x_0)>0$, then, recalling that
			$(\lambda_{p,q})^\frac{1}{q-2}\to 
			(\Lambda_\infty(\Gamma,Q))^{\nicefrac{1}{Q}}$, it follows 
			that we must have 
			$\frac{\Lambda_\infty(\Gamma,Q)^{\nicefrac{1}{Q}} 
			|u_\infty(x_0)|^{\nicefrac{\Gamma}{Q}} 
			|v_\infty(x_0)|^{1-\Gamma} }{|\nabla\phi(x_0)|}
			\le 1$. Going back to 
			\eqref{Eq3Visc} we also get $-\Delta_\infty\phi(x_0)\ge 0$.

			If $v_\infty(x_0)<0$ then we rewrite the equation as
			\begin{align*}
				-|\nabla\phi (x_{p,q})|^{-2}
				\left(\frac{(q-2)^\frac{1}{q-2}
				|\nabla\phi(x_{p,q})|}
				{\lambda_{p,q}^\frac{1}{q-2}
				|u_{p,q}(x_{p,q})|^{\frac{\alpha}{q-2}}
				|v_{p,q}(x_{p,q})|^{\frac{\beta-2}{q-2}}}\right)^{q-2}
 				&(- \Delta_\infty\phi(x_0)+o(1))\\ 
 				&\le 
 				-v_{p,q}(x_{p,q}). 
			\end{align*}
			If $ \frac{\Lambda_\infty(\Gamma,Q)^{\nicefrac{1}{Q}} 
			|u_\infty(x_0)|^{\nicefrac{\Gamma}{Q}} 
			|v_\infty(x_0)|^{1-\Gamma} }{|\nabla\phi(x_0)|}< 1$ 
			then we must have $- \Delta_\infty\phi(x_0)\ge 0$. 
			Otherwise we have $ \frac{\Lambda_\infty
			(\Gamma,Q)^{\nicefrac{1}{Q}} 
			|u_\infty(x_0)|^{\nicefrac{\Gamma}{Q}} 
			|v_\infty(x_0)|^{1-\Gamma}}{|\nabla\phi(x_0)|}
			\geq 1$.

			If $v_\infty(x_0)=0$, then $v_{p,q}(x_{p,q})\to 0$ 
			so that $|v_{p,q}(x_{p,q})|^{q-2}v_{p,q}(x_{p,q})\le 
			v_{p,q}(x_{p,q})\to 0$.
			It then follows that
			$$ 
				-|\nabla\phi(x_{p,q})|^{q-2}\Delta\phi(x_{p,q})-(q-2)
				|\nabla\phi(x_{p,q})|^{q-4}\Delta_\infty
				\phi(x_{p,q})\ge o(1).
			$$
			Dividing this inequality by 
			$(q-2)|\nabla\phi(x_{p,q})|^{q-4}$ and letting 
			$p,q\to \infty$ we obtain $- \Delta_\infty\phi(x_0)\ge 0$.

			Assume now that $x_0\in \partial \Omega$.
			We have to prove that
			$$ 
				\max\,\left\{ F_\infty (x_0,\phi(x_0),
				\nabla\phi(x_0),D^2\phi(x_0)), 
				\dfrac{\partial\phi}{\partial\nu}(x_0)
				 \right\}\ge 0.
			$$
			If $x_{p,q}\in \Omega $ for some subsequence then we can 
			proceed as before to get 
			$$
				F_\infty (x_0,\phi(x_0),
				\nabla\phi(x_0),D^2\phi(x_0)) \ge 0.
			$$
			Assume that $x_{p,q}\in \partial \Omega$ 
			for every $p,q$ large. 
			If $\nabla \phi(x_0)=0$ then 
			$\nicefrac{\partial\phi(x_0)}{\partial\nu}=0$. 
			Then we need to deal with $\nabla \phi(x_0)\neq 0$.
			We have 
			$$ 
				\max\,\left\{ F_p (x_{p,q},\phi(x_{p.q}),
				\nabla\phi(x_{p,q}),
				D^2\phi(x_{p,q})), 
				\dfrac{\partial\phi}{\partial\nu} (x_{p,q})
				\right\}\geq 0.
			$$
			If $F_p (x_{p,q},\phi(x_{p,q}),
			\nabla\phi(x_{p,q}),D^2\phi(x_{p,q})) \geq 0$ 
			holds for a subsequence we are done as before. Otherwise
			$$ 
				\dfrac{\partial\phi}{\partial\nu}(x_{p,q})\ge 0 \qquad 
				\text{for $p,q$ large} 
			$$
			so that $\nicefrac{\partial\phi}{\partial\nu}(x_0)=\lim_{p,q\to \infty}
			\nicefrac{\partial\phi}{\partial\nu}(x_{p,q}) \ge 0$.
\end{proof}

\end{document}